\documentclass{amsart}

\usepackage{amsmath,amsthm,amssymb,mathtools,amscd,color,comment}
\usepackage{cleveref}
\usepackage{tikz}
\usetikzlibrary{intersections, calc, arrows.meta}
\usepackage[all]{xy}
\usepackage{enumerate}
\usepackage{url}
\usepackage{setspace}
\allowdisplaybreaks

\newtheorem{thm}{Theorem}[section]
\newtheorem{prop}[thm]{Proposition}
\newtheorem{lem}[thm]{Lemma}
\newtheorem{cor}[thm]{Corollary}
\newtheorem{rem}[thm]{Remark}
\newtheorem{conj}[thm]{Conjecture}
\newtheorem{prob}[thm]{Problem}
\theoremstyle{definition}

\newtheorem*{thmA}{Theorem A}
\newtheorem*{thmB}{Theorem B}

\DeclareMathOperator{\Harm}{Harm}

\DeclareMathOperator{\Hst}{Hst}

\newcommand{\RR}{\mathbb{R}}
\newcommand{\CC}{\mathbb{C}}
\newcommand{\NN}{\mathbb{N}}

\title{Spherical Designs on $S^1$
 of Finite Harmonic Strength}

\author[R. Misawa]{Ryutaro Misawa}
\address[R.~Misawa]{Division of Mathematics\\
Graduate School of Information Sciences\\
Tohoku University\\
6-3-09 Aramaki-Aza-Aoba, Aoba-ku, Sendai 980-8579\\
Japan}
\email{misawa.ryutaro.q2@dc.tohoku.ac.jp}
\thanks{Corresponding author: Ryutaro Misawa
(\texttt{misawa.ryutaro.q2@dc.tohoku.ac.jp}).}

\author[Y. Nishimura]{Yusaku Nishimura}
\address[Y. Nishimura]{Graduate School of Fundamental Science and Engineering,
Waseda University, Tokyo 169-8555, Japan}
\email{n2357y@ruri.waseda.jp}

\subjclass[2020]{Primary 05B30; Secondary 42C10, 52C17, 65D32.}

\keywords{spherical designs, harmonic strength, harmonic index designs, unit circle}

\date{\today}

\begin{document}

\begin{abstract}
We study exact harmonic strengths of finite spherical designs on the unit circle.  For a nonempty finite set \(X\subset S^1\), let \(\Hst(X)\) be the set of positive integers \(k\) for which the \(k\)-th complex moment \(P_k(X)=\sum_{x\in X}x^k\) vanishes.
Equivalently, \(X\) is a spherical \(T\)-design precisely when \(T\subset \Hst(X)\).  We consider the exact realization problem: given a finite set \(T\subset\mathbb N\), determine whether there exists a finite set \(X\subset S^1\) such that \(\Hst(X)=T\).
We prove that every finite \(T\subset\mathbb N\) is realizable. More precisely, for each \(t\ge 1\) we construct uncountably many five-point sets with \(\Hst(X)=\{t\}\), and we prove that no smaller set can have this exact harmonic strength.  A product construction then gives, for every finite \(T\subset\mathbb N\), a realization with
\(|X|=5^{|T|}\).  We also initiate the associated minimum-size problem \(N(T,2)\).  We prove \(N(\{t\},2)=5\) for all \(t\ge1\), determine
\(N(\{2,3\},2)=5\), and show that the optimal \(\{2,3\}\)-example is unique up to rotation.  Finally, we discuss a rigid seven-point example related to \(T=\{2,3,4,10\}\).

\end{abstract}

\maketitle

\section{Introduction}\label{sec:intro}

Throughout, we write \(\NN=\{1,2,3,\ldots\}\).
Spherical designs, introduced by Delsarte, Goethals, and Seidel \cite{DGS77},
are finite sets of points on the unit sphere \(S^{d-1}\subset\RR^d\) that
reproduce exact surface averages of low-degree polynomials.  Concretely, a
non-empty finite set \(X\subset S^{d-1}\) is a \emph{spherical \(t\)-design} if
\begin{equation}
   \label{eq:cubature}
  \frac{1}{|S^{d-1}|}\int_{S^{d-1}} f(\xi)\,d\xi
  =
  \frac{1}{|X|}\sum_{\xi\in X}f(\xi)
  \quad
  \bigl(\deg f\le t\bigr).
\end{equation}
When \eqref{eq:cubature} holds for every polynomial of degree at most \(t\)
but fails for some polynomial of degree \(t+1\), the integer \(t\) is called
the \emph{strength} of \(X\).

Classical work on spherical designs has been driven by the following questions:
\begin{enumerate}
    \item[(a)] For which parameters does no design exist?
    \item[(b)] For which parameters does some design exist?
    \item[(c)] How can explicit examples be constructed?
\end{enumerate}
The unit circle \(S^1\) is a useful testing ground for these questions because
of its simple trigonometric parametrisation.  Hong \cite{Hong82} classified
one-dimensional spherical \(t\)-designs into group-type and non-group type
families, and Martínez \cite{Martinez22} sharpened this classification by using
symmetric function techniques.

A natural generalisation replaces the initial segment \(\{1,\ldots,t\}\) by an
arbitrary set of degrees.  For \(T\subset\NN\), we call a finite set
\(X\subset S^{d-1}\) a \emph{spherical \(T\)-design} if
\[
  \sum_{\xi\in X}P(\xi)=0
  \quad
  \bigl(P\in\Harm_k(d),\ k\in T\bigr),
\]
where \(\Harm_k(d)\) denotes the space of real homogeneous harmonic polynomials
of degree \(k\) in \(d\) variables.  The \emph{harmonic strength} of \(X\) is
\[
  \Hst(X):=
  \left\{
  k\in\NN
  \,\middle|\,
  \sum_{\xi\in X}P(\xi)=0
  \text{ for all }P\in\Harm_k(d)
  \right\}.
\]

The exact condition considered in this paper is different from the usual
strength condition.  The usual spherical \(t\)-design condition concerns the
initial segment \(\{1,\ldots,t\}\), whereas \(\Hst(X)\) records the full set of
degrees for which all harmonic moments vanish.  Thus the harmonic strength of
\(X\) need not be an initial segment, and it is not determined by the ordinary
strength alone.  Consequently, even on \(S^1\), the problem of prescribing
\(\Hst(X)\) exactly is not a formal consequence of the existing classification
of spherical \(t\)-designs.

In this paper, we focus on the exact realisation problem
\[
  \Hst(X)=T.
\]
For a given finite set \(T\subset\NN\), the following questions naturally arise:
\begin{enumerate}\setlength{\itemsep}{0pt}
    \item[(a')] For which \(T\) does no finite set \(X\) with \(\Hst(X)=T\) exist?
    \item[(b')] Does such an \(X\) exist for every finite set \(T\subset\NN\)?
    \item[(c')] If such an \(X\) exists, can one construct it explicitly?
\end{enumerate}
These questions are already nontrivial in dimension \(d=2\).  Miezaki
\cite{Miezaki13} determined the harmonic strength for the shells of the lattice
\(\mathbb Z^2\).  Bannai et al.\ \cite{ZBBKY17} studied non-existence results
for spherical designs of prescribed harmonic indices in higher dimensions.  On
\(S^1\), however, the exact realisation problem \(\Hst(X)=T\) for arbitrary
finite \(T\subset\NN\) had not been settled in this generality.

Our first result treats the case of a singleton \(T=\{t\}\).

\begin{thmA}\label{thm:A}
For every integer \(t\ge1\), there exist uncountably many \(5\)-point sets
\(X\subset S^1\) such that \(\Hst(X)=\{t\}\) and \(1\in X\).  No subset of
\(S^1\) with fewer points has this property.
\end{thmA}

Using Theorem~A and a product construction on \(S^1\), we then prove the
existence result for arbitrary finite \(T\).

\begin{thmB}\label{thm:B}
For every finite set \(T\subset\NN\), there exists a finite set \(X\subset S^1\)
such that \(\Hst(X)=T\) and \(|X|=5^{|T|}\).
\end{thmB}

We also initiate the associated minimum-size problem.  For a nonempty finite set
\(T\subset\NN\), define
\[
  N(T,2):=
  \min\bigl\{|X| \,\bigm|\, X\subset S^1,\ \Hst(X)=T\bigr\}.
\]
In Section~3, we prove \(N(\{t\},2)=5\) for every \(t\ge1\), determine
\(N(\{2,3\},2)=5\), and show that the optimal \(\{2,3\}\)-example is unique up
to rotation.  We also exhibit a seven-point set \(X_0\subset S^1\) with
\[
  \{2,3,4,10\}\subset\Hst(X_0),
\]
formulate the conjecture \(\Hst(X_0)=\{2,3,4,10\}\), and prove a corresponding
conditional uniqueness result.

Finally, we identify \(S^1\) with \(\{z\in\CC\mid |z|=1\}\) and write
\[
  P_k(X):=\sum_{x\in X}x^k
\]
for the \(k\)-th complex moment of a finite set \(X\subset S^1\).  By
\cite[Lemma~1]{Hong82}, for \(X\subset S^1\) one has
\[
  \Hst(X)=\{k\in\NN\mid P_k(X)=0\}.
\]
\section{Construction of Designs with Prescribed Harmonic Strength}
\label{sec:hst_T}

In this section, we prove Theorem A and Theorem B.
First, we prove some lemmas. The following lemma is well-known as the Identity Theorem (see ~\cite[Corollary 1.2.7]{KP02}).

\begin{lem}\label{lem:zeroP}
  Let $I\subset\RR$ be an open interval, and let $f:I\to\RR$ be a real-analytic function. If its zero set
  \[Z(f)=\{x\in I\mid f(x)=0\}\]
  has an accumulation point in $I$, then $f$ is identically zero on $I$. In particular, nontrivial real-analytic functions have only countably many zeros.
\end{lem}

\begin{proof}
The first statement is the classical Identity Theorem for real-analytic functions. ~\cite[Corollary 1.2.7]{KP02}.

The second statement follows from this isolation property. Indeed, since $I$ is a countable union of compact subsets, $f$ has only countably many zeros.
\end{proof}

\begin{lem}\label{lem:singleton}
    Let $D$ be any open interval satisfying
    \[\emptyset\neq D\subseteq\Bigl(-1,\tfrac12\Bigr)\setminus\{-\tfrac14\},\]
  and for each $r\in\mathbb{Q}_{>0}$ define
  \[f_r(x)=1+2\cos\bigl(r\arccos(x)\bigr)+2\cos\bigl(r\arccos(-x-\tfrac12)\bigr).
  \]
  Then
  \[\{x\in D\mid \forall r\in\mathbb{Q}_{>0}\setminus\{1\}:f_r(x)\neq0\}\neq\emptyset.
  \]
  Moreover, this set is uncountable.
\end{lem}

\begin{proof}
Let \(I:=(-1,\tfrac12)\). Since \(\arccos(x)\) and
\(\arccos(-x-\tfrac12)\) are analytic on \(I\), each \(f_r\) is analytic
on \(D\).

We first note that \(f_r\) is not identically zero on \(D\) for
\(r\neq1\). Indeed, suppose that \(f_r\equiv0\) on \(D\). Then, by
Lemma~\ref{lem:zeroP}, \(f_r\equiv0\) on \(I\). Put
\(c=-\tfrac14\) and \(h(x)=\cos(r\arccos x)\). Since
\(-c-\tfrac12=c\), the equalities \(f_r(c)=0\) and \(f_r''(c)=0\) give
\(h(c)=-\tfrac14\) and \(h''(c)=0\). Also \(h\) satisfies
\[
(1-x^2)h''(x)-xh'(x)+r^2h(x)=0.
\]
Substituting \(x=c\), we get \(h'(c)=r^2\). On the other hand,
\[
h'(c)^2
=
\frac{r^2(1-h(c)^2)}{1-c^2}
=
r^2.
\]
Hence \(r^4=r^2\), and since \(r>0\), we obtain \(r=1\), a contradiction.

For \(r\neq1\), let \(Z_r\) be the set of zero points of \(f_r\) that
are also in \(D\). Then, from Lemma~\ref{lem:zeroP}, \(Z_r\) is countable.
Therefore,
\[
Z=\bigcup_{r\in\mathbb{Q}_{>0}\setminus\{1\}} Z_r
\]
is also countable, while \(D\) is uncountable. Thus \(D\setminus Z\) is
also uncountable, yielding uncountably many \(x\in D\) such that
\(f_r(x)\neq0\) for all \(r\in\mathbb{Q}_{>0}\setminus\{1\}\).
\end{proof}
 

\begin{lem}\label{lem:singleton2}
  Let $D$ be any open interval satisfying
    \[\emptyset\neq D\subseteq\Bigl(-1,\tfrac12\Bigr)\setminus\{-\tfrac14\}.\]
  For any $x\in D$, define
\[
    X^{\frac{1}{t}}(x)\coloneqq 
\left\{1,e^{{i\arccos(x)}/{t}},e^{-{i\arccos(x)}/{t}},
e^{i\arccos\left(-x-\tfrac12\right)/t},
e^{-i\arccos\left(-x-\tfrac12\right)/t}
\right\},
\]
where $t$ is any natural number.
  Then, there exist uncountably many $x\in D$ such that $\Hst(X^{\frac{1}{t}}(x))=\{t\}$.
\end{lem}

\begin{proof}
      Since $x\notin\{-1,-\tfrac14,\tfrac12\}$, $X^{\tfrac1t}(x)$ contains $5$ distinct elements.
      Write $X(x)=X^1(x)$.
  Since
  \( P_l(X(x))=f_l(x), \)
  we have $P_1(X(x))=0$ for any $x \in D$. 
  From \Cref{lem:singleton}, there exist uncountably many $x \in D$ such that
$f_r(x)\neq0$ for any $r\in\mathbb{Q}_{>0}\setminus\{1\}$.
For these $x$, since $P_k(X^{\frac{1}{t}}(x))=f_{\frac{k}{t}}(x)$, $P_k(X^{\tfrac1t}(x))=0$ if and only if $k=t$. 
Then, the harmonic strength of $X^{\frac{1}{t}}(x)$ is $\{t\}$.
\end{proof}

Now, we can prove Theorem A.

\begin{proof}[Proof of Theorem A]
  From \Cref{lem:singleton2}, there exist uncountably many \(5\)-point sets \(X\subset S^1\) with \(\Hst(X)=\{t\}\) and \(1\in X\).
  
 Next, we show that no subset with fewer points has harmonic strength \(\{t\}\). Assume that \(X\subset S^1\) satisfies \(\Hst(X)=\{t\}\) and \(|X|\le4\). Let
\[
Y:=\{\xi^t\mid \xi\in X\}
\]
be counted as a multiset. Then
\[
P_1(Y)=P_t(X)=0.
\]
Thus \(Y\) is a zero-sum multiset of at most four points on \(S^1\).
Such a multiset is either an antipodal pair, a regular triangle, or a
union of two antipodal pairs, with multiplicities allowed. Hence
\(P_k(Y)=0\) for infinitely many \(k\). Since \(P_k(Y)=P_{tk}(X)\), the
set \(\Hst(X)\) is infinite. This contradicts \(\Hst(X)=\{t\}\).
Therefore \(|X|\ge5\).
\end{proof}
  
\begin{rem}
  The $X(x)$ constructed in \Cref{lem:singleton2} has $5$ points. It is known that the minimum number of points for a spherical $t$-design on $S^1$ is $t+1$, and for $t=1$ this minimum is two. 
  Our construction gives a non-group type $1$-design with $n=5$, matching the lower bound $n\ge2t+3$ in Hong's result (Theorem A of \cite{Hong82}).
\end{rem}

Next, we prove Theorem B. 
To achieve this, we define a combination of spherical designs on $S^1$.
First, we show the following lemma.
Note that since we identify \(S^1\) with a multiplicative subgroup of $\mathbb{C}$, for any subsets $X_1, X_2 \subset S^1$, their product $X_1 \cdot X_2$ is defined as follows:
\[
X_1 \cdot X_2 = \{cd \mid c \in X_1,\, d \in X_2\}.
\]

\begin{lem}
\label{lem:prod}
If $X_1,X_2$ have harmonic strengths $T_1,T_2$ respectively, and $|X_1\cdot X_2|=|X_1||X_2|$, then $X_1\cdot X_2$ has harmonic strength $T_1\cup T_2$.
\end{lem}

\begin{proof}
  Direct calculation gives:
  \begin{align*}
   P_k(X_1\cdot X_2)&=P_k(X_1)P_k(X_2).
  \end{align*} 
  Therefore, $P_k(X_1\cdot X_2)=0$ iff $P_k(X_1)=0$ or $P_k(X_2)=0$, yielding the result.
\end{proof}
\begin{proof}[Proof of Theorem B]
We prove that there exists a subset \(X\subset S^1\) such that
\(\Hst(X)=T\) and \(|X|=5^{|T|}\) by induction on the size of \(T\).
If \(T=\emptyset\), take \(X=\{1\}\). Then \(\Hst(X)=\emptyset\).
When \(|T|=1\), from Theorem~A, there exists a subset \(X_1\subset S^1\)
such that \(|X_1|=5\) and \(\Hst(X_1)=T\).

Let \(T'\) be any subset of \(\NN\) with \(|T'|=k+1\geq2\), and let
\(t\) be any element of \(T'\). By the inductive hypothesis, there exists
a subset \(X_k\subset S^1\) such that \(\Hst(X_k)=T'\setminus\{t\}\) and
\(|X_k|=5^k\).

Recall that
\begin{align*}
X(x)
&=\left\{
1,e^{i\arccos(x)},e^{-i\arccos(x)},
e^{i\arccos(-x-\tfrac12)},e^{-i\arccos(-x-\tfrac12)}
\right\} \\
&=\left\{
1,x\pm i\sqrt{1-x^2},
-x-\tfrac12\pm i\sqrt{-x^2-x+\tfrac34}
\right\} \\
&=\{g_1(x),g_2(x),\ldots,g_5(x)\}
\end{align*}
constructed in \Cref{lem:singleton2} for
\(x\in(-1,\tfrac12)\setminus\{-\tfrac14\}\). Define
\[
D_k\coloneqq
\left\{
x\in(-1,\tfrac12)\setminus\{-\tfrac14\}
\,\middle|\,
|X_k\cdot X^{\tfrac1t}(x)|
<
|X_k|\,|X^{\tfrac1t}(x)|
\right\}.
\]

We show that \(D_k\) is finite. Write
\(X^{\tfrac1t}(x)=\{h_1(x),\ldots,h_5(x)\}\), where \(h_i(x)^t=g_i(x)\).
If \(x\in D_k\), then there exist \(a,b\in X_k\) and \(i,j\) with
\(a h_i(x)=b h_j(x)\), not coming from the same pair. Since
\(X^{\tfrac1t}(x)\) has five distinct elements, we must have \(i\neq j\).
Putting \(\xi=b/a\in X_k\cdot X_k^{-1}\), we get
\[
h_i(x)\overline{h_j(x)}=\xi,
\]
and hence
\[
g_i(x)\overline{g_j(x)}=\xi^t.
\]
Therefore
\[
D_k\subseteq
\bigcup_{\substack{1\leq i,j\leq5\\ i\neq j}}
\bigcup_{\xi\in X_k\cdot X_k^{-1}}
\left\{
x\in(-1,\tfrac12)\setminus\{-\tfrac14\}
\,\middle|\,
g_i(x)\overline{g_j(x)}-\xi^t=0
\right\}.
\]
For \(i\neq j\), the function \(g_i\overline{g_j}\) is a nonconstant algebraic function. Hence each zero set in the above union is finite.
Since the union is finite, \(D_k\) is finite.

Choose a nonempty open interval
\[
D\subset (-1,\tfrac12)\setminus\bigl(D_k\cup\{-\tfrac14\}\bigr).
\]
By \Cref{lem:singleton2}, there exists \(x\in D\) such that
\(\Hst(X^{\tfrac1t}(x))=\{t\}\) and \(|X^{\tfrac1t}(x)|=5\).
Since \(x\notin D_k\), we also have
\[
|X_k\cdot X^{\tfrac1t}(x)|
=
|X_k|\,|X^{\tfrac1t}(x)|.
\]
From \Cref{lem:prod},
\[
\Hst\bigl(X_k\cdot X^{\tfrac1t}(x)\bigr)=T'
\]
and
\[
|X_k\cdot X^{\tfrac1t}(x)|=5^{k+1}.
\]
Therefore, for any \(T'\subset\NN\) with \(|T'|=k+1\), there exists a
subset \(X\subset S^1\) such that \(\Hst(X)=T'\) and \(|X|=5^{k+1}\).
Thus, by induction, for any finite subset \(T\) of \(\NN\), there exists
a subset \(X\subset S^1\) such that \(\Hst(X)=T\) and \(|X|=5^{|T|}\).
\end{proof}
\section{The minimum size problem on \(S^1\)}

We now consider the minimum number of points needed to realize a prescribed
finite harmonic strength.  For a nonempty finite set \(T\subset\NN\), define
\[
N(T,2):=\min\bigl\{|X| \,\bigm|\, X\subset S^1,\ \Hst(X)=T\bigr\}.
\]
By Theorems~A and~B,
\[
N(\{t\},2)=5\quad(t\ge1),
\qquad
N(T,2)\le 5^{|T|}.
\]
The following problem is therefore natural.

\begin{prob}
Determine \(N(T,2)\) for finite nonempty \(T\subset\NN\).
\end{prob}

\begin{prop}\label{prop:NT_ge_5}
Let \(T\subset\NN\) be finite and nonempty. Then
\[
N(T,2)\ge5.
\]
\end{prop}

\begin{proof}
Let \(X\subset S^1\) satisfy \(\Hst(X)=T\), and choose \(t\in T\).
Let
\[
Y:=\{\xi^t\mid \xi\in X\}
\]
be counted as a multiset. Then
\[
P_1(Y)=P_t(X)=0.
\]
If \(|X|\le4\), then \(Y\) is a zero-sum multiset of at most four points
on \(S^1\). Such a multiset is either an antipodal pair, a regular triangle,
or a union of two antipodal pairs, with multiplicities allowed. Hence
\(P_k(Y)=0\) for infinitely many \(k\). Since
\[
P_k(Y)=P_{tk}(X),
\]
the set \(\Hst(X)\) is infinite. This contradicts \(\Hst(X)=T\), where
\(T\) is finite. Therefore \(|X|\ge5\).
\end{proof}

\begin{conj}\label{conj:double}
Let \(p\neq q\) be integers with \(p,q>1\). Then
\[
N(\{p,q\},2)=5.
\]
\end{conj}

The first case is as follows.

\begin{prop}\label{prop:23}
\[
N(\{2,3\},2)=5.
\]
\end{prop}

\begin{proof}
By Proposition~\ref{prop:NT_ge_5}, it is enough to construct a five-point
set.  Let
\[
X=\{1,e^{\pm ia},e^{\pm ib}\},
\qquad
x=\cos a,\quad y=\cos b.
\]
Then
\[
P_k(X)=1+2T_k(x)+2T_k(y),
\]
where \(T_k\) is the Chebyshev polynomial.  The equations
\(P_2(X)=P_3(X)=0\) are
\[
x^2+y^2=\frac34,
\qquad
8(x^3+y^3)-6(x+y)+1=0.
\]
Put \(s=x+y\) and \(p=xy\). Then
\[
p=\frac{s^2-\frac34}{2},
\qquad
(s-1)(2s+1)^2=0.
\]
Since \(P_1(X)=1+2s\), we take \(s=1\). Hence \(p=1/8\), and we may take
\[
x=\frac12-\frac{\sqrt2}{4},
\qquad
y=\frac12+\frac{\sqrt2}{4}.
\]
Then \(P_1(X)\neq0\) and \(P_2(X)=P_3(X)=0\).

It remains to show that there are no further vanishing moments. Put
\[
q_k:=P_k(X)-1.
\]
The four numbers \(e^{\pm ia},e^{\pm ib}\) are the roots of
\[
2z^4-4z^3+5z^2-4z+2.
\]
Therefore
\[
2q_{k+4}-4q_{k+3}+5q_{k+2}-4q_{k+1}+2q_k=0
\qquad(k\ge0).
\]
Moreover,
\[
q_0=4,\quad q_1=2,\quad q_2=q_3=-1,\quad q_4=\frac12,\quad q_5=-\frac12.
\]
For \(m\ge1\), define
\[
E_m:=2^{m-1}q_{2m},
\qquad
O_m:=2^{m-1}q_{2m+1}.
\]
Putting \(k=2m\) and \(k=2m+1\) in the recurrence gives
\begin{align*}
E_{m+2}-4O_{m+1}+5E_{m+1}-8O_m+4E_m&=0,\\
O_{m+2}-2E_{m+2}+5O_{m+1}-4E_{m+1}+4O_m&=0.
\end{align*}
Since
\[
E_1=O_1=-1,\qquad E_2=1,\qquad O_2=-1,
\]
these relations show by induction that \(E_m,O_m\in\mathbb Z\) for all
\(m\ge1\). Reducing them modulo \(2\), we obtain
\[
E_{m+2}\equiv E_{m+1},
\qquad
O_{m+2}\equiv O_{m+1}
\pmod 2.
\]
Hence \(E_m\) and \(O_m\) are odd for all \(m\ge2\).
Thus, for \(k\ge4\), the number \(q_k\) is not an integer. In particular
\(q_k\neq-1\). Since \(P_k(X)=0\) is equivalent to \(q_k=-1\), we get
\[
P_k(X)\neq0\qquad(k\ge4).
\]
Hence
\[
\Hst(X)=\{2,3\}.
\]
\end{proof}

\begin{rem}
For a general pair \(\{p,q\}\), solving the equations
\(P_p(X)=P_q(X)=0\) for a five-point set does not by itself prove
\(\Hst(X)=\{p,q\}\).  One must also exclude all further vanishing moments.
\end{rem}

We shall also use the following elementary fact.

\begin{lem}\label{lem:unit-circle-coefficients}
Let \(X=\{\xi_1,\dots,\xi_n\}\subset S^1\), and write
\[
F_X(z):=\prod_{j=1}^n(z-\xi_j)
=z^n-e_1z^{n-1}+e_2z^{n-2}-\cdots+(-1)^ne_n.
\]
Then \(|e_n|=1\), and
\[
e_m=e_n\overline{e_{n-m}}
\qquad(0\le m\le n),
\]
where \(e_0=1\).
\end{lem}

\begin{proof}
Since \(\xi_j^{-1}=\overline{\xi_j}\), we have
\[
e_n\overline{e_{n-m}}
=
\left(\prod_{j=1}^n\xi_j\right)
\sum_{|I|=n-m}\prod_{i\in I}\overline{\xi_i}
=
\sum_{|J|=m}\prod_{j\in J}\xi_j
=
e_m.
\]
Also \(|e_n|=\prod_j|\xi_j|=1\).
\end{proof}

\begin{prop}\label{prop:unique23}
A set \(X\subset S^1\) with \(|X|=N(\{2,3\},2)\) and
\(\Hst(X)=\{2,3\}\) is unique up to rotation.
\end{prop}

\begin{proof}
Let \(X\subset S^1\) be such a set.  Put \(s_k=P_k(X)\).  Then
\[
|X|=5,\qquad s_2=s_3=0,\qquad s_1\neq0.
\]
Write
\[
F_X(z)=z^5-e_1z^4+e_2z^3-e_3z^2+e_4z-e_5.
\]
Newton's identities give
\[
e_1=s_1,\qquad
e_2=\frac{s_1^2}{2},\qquad
e_3=\frac{s_1^3}{6}.
\]
By Lemma~\ref{lem:unit-circle-coefficients},
\[
e_3=e_5\overline{e_2}.
\]
Taking absolute values gives
\[
\left|\frac{s_1^3}{6}\right|
=
\left|\frac{s_1^2}{2}\right|,
\]
so \(|s_1|=3\).  After rotation, we may assume \(s_1=3\).  Then
\[
e_1=3,\qquad e_2=e_3=\frac92.
\]
Again by Lemma~\ref{lem:unit-circle-coefficients}, we get
\[
e_5=1,\qquad e_4=3.
\]
Thus
\[
F_X(z)=z^5-3z^4+\frac92z^3-\frac92z^2+3z-1=\frac{1}{2}(2z^4-4z^3+5z^2-4z+2)(z-1).
\]
Hence \(X\) is determined up to rotation.
\end{proof}

We next consider \(T=\{2,3,4,10\}\).

\begin{prop}\label{prop:23410-lower}
\[
N(\{2,3,4,10\},2)\ge7.
\]
\end{prop}

\begin{proof}
Let \(X\subset S^1\) satisfy \(\Hst(X)=\{2,3,4,10\}\), and put
\(s_k=P_k(X)\).  Then
\[
s_2=s_3=s_4=s_{10}=0,\qquad s_1\neq0.
\]
By Proposition~\ref{prop:NT_ge_5}, it remains to exclude \(|X|=5,6\).

Suppose first that \(|X|=5\).  Write
\[
F_X(z)=z^5-e_1z^4+e_2z^3-e_3z^2+e_4z-e_5.
\]
Newton's identities give
\[
e_1=s_1,\quad
e_2=\frac{s_1^2}{2},\quad
e_3=\frac{s_1^3}{6},\quad
e_4=\frac{s_1^4}{24}.
\]
By Lemma~\ref{lem:unit-circle-coefficients},
\[
|e_3|=|e_2|,
\qquad
|e_4|=|e_1|.
\]
Hence
\[
|s_1|=3,\qquad |s_1|^3=24,
\]
a contradiction.

Suppose next that \(|X|=6\).  Write
\[
F_X(z)=z^6-e_1z^5+e_2z^4-e_3z^3+e_4z^2-e_5z+e_6.
\]
Again,
\[
e_1=s_1,\quad
e_2=\frac{s_1^2}{2},\quad
e_3=\frac{s_1^3}{6},\quad
e_4=\frac{s_1^4}{24}.
\]
Lemma~\ref{lem:unit-circle-coefficients} gives \(|e_4|=|e_2|\), hence
\[
|s_1|^2=12.
\]
After rotation, assume \(s_1=\sqrt{12}\).  Then
\[
e_1=e_5=\sqrt{12},
\qquad
e_2=e_4=6,
\qquad
e_3=2\sqrt{12},
\qquad
e_6=1.
\]
Put \(a=\sqrt{12}\). The recurrence obtained from \(F_X\) is
\[
s_{k+6}-a s_{k+5}+6s_{k+4}-2a s_{k+3}
+6s_{k+2}-a s_{k+1}+s_k=0
\qquad(k\ge0).
\]
Using
\[
s_0=6,\qquad s_1=a,\qquad s_2=s_3=s_4=s_{10}=0,
\]
the cases \(k=0,1,2,3,4\) give
\[
\begin{aligned}
s_6&=a s_5+6, &
s_7&=6s_5+5a,\\
s_8&=2a s_5+24, &
s_9&=6s_5+6a,
\end{aligned}
\qquad
a s_5+12=0.
\]
Thus \(s_5=-a\), and hence
\[
s_8=s_9=0.
\]
This contradicts \(\Hst(X)=\{2,3,4,10\}\).  Therefore \(|X|\ge7\).
\end{proof}

\begin{prop}\label{prop:23410-example}
Let
\[
F_0(z)=
z^7-4z^6+8z^5-\frac{32}{3}z^4
+\frac{32}{3}z^3-8z^2+4z-1,
\]
and let \(X_0\) be the set of zeros of \(F_0\). Then
\[
X_0\subset S^1,\qquad |X_0|=7,
\qquad
\{2,3,4,10\}\subset \Hst(X_0).
\]
\end{prop}

\begin{proof}
We have
\[
F_0(z)=(z-1)
\left(
z^6-3z^5+5z^4-\frac{17}{3}z^3+5z^2-3z+1
\right).
\]
For the second factor, put \(u=z+z^{-1}\).  Then
\[
z^{-3}
\left(
z^6-3z^5+5z^4-\frac{17}{3}z^3+5z^2-3z+1
\right)
=
u^3-3u^2+2u+\frac13.
\]
Thus \(u\) satisfies
\[
3u^3-9u^2+6u+1=0.
\]
This cubic has one root in each of the intervals
\[
(-2,0),\qquad (1,3/2),\qquad (3/2,2).
\]
Hence all its roots lie in \((-2,2)\), and the corresponding six zeros
of the second factor lie on \(S^1\).  The remaining zero is \(1\).
Thus \(X_0\subset S^1\) and \(|X_0|=7\).

Let \(s_k=P_k(X_0)\). Newton's identities applied to \(F_0\) give
\[
s_1=4,\qquad
s_2=s_3=s_4=0,\qquad
s_5=s_6=-\frac83,\qquad
s_7=\frac53.
\]
The recurrence obtained from \(F_0\) then gives
\[
s_8=\frac{32}{9},\qquad
s_9=\frac89,\qquad
s_{10}=0.
\]
Therefore
\[
2,3,4,10\in\Hst(X_0).
\]
\end{proof}

\begin{conj}\label{conj:23410}
For the set \(X_0\) in Proposition~\ref{prop:23410-example},
\[
\Hst(X_0)=\{2,3,4,10\}.
\]
Equivalently,
\[
P_k(X_0)\neq0
\qquad
(k\notin\{2,3,4,10\}).
\]
\end{conj}

\begin{cor}\label{cor:23410-conditional}
Assume Conjecture~\ref{conj:23410}. Then
\[
N(\{2,3,4,10\},2)=7.
\]
\end{cor}

\begin{proof}
The lower bound is Proposition~\ref{prop:23410-lower}.  The upper bound
follows from Proposition~\ref{prop:23410-example} and
Conjecture~\ref{conj:23410}.
\end{proof}

\begin{prop}\label{prop:23410-rigidity}
Let \(X\subset S^1\) be a seven-point set such that
\[
P_2(X)=P_3(X)=P_4(X)=0,
\qquad
P_1(X)\neq0.
\]
Then \(X\) is equal to \(X_0\) up to rotation.
\end{prop}

\begin{proof}
Put \(s_k=P_k(X)\), and write
\[
F_X(z)=z^7-e_1z^6+e_2z^5-e_3z^4+e_4z^3-e_5z^2+e_6z-e_7.
\]
Newton's identities give
\[
e_1=s_1,\quad
e_2=\frac{s_1^2}{2},\quad
e_3=\frac{s_1^3}{6},\quad
e_4=\frac{s_1^4}{24}.
\]
By Lemma~\ref{lem:unit-circle-coefficients},
\[
e_3=e_7\overline{e_4}.
\]
Taking absolute values gives \(|s_1|=4\).  After rotation, assume
\(s_1=4\).  Then
\[
e_1=4,\qquad e_2=8,\qquad e_3=e_4=\frac{32}{3}.
\]
The coefficient relations in Lemma~\ref{lem:unit-circle-coefficients} give
\[
e_7=1,\qquad e_6=4,\qquad e_5=8.
\]
Therefore
\[
F_X(z)=F_0(z).
\]
Hence \(X=X_0\) after rotation.
\end{proof}

\begin{cor}\label{cor:23410-unique-conditional}
Assume Conjecture~\ref{conj:23410}. Then a set \(X\subset S^1\) with
\[
|X|=N(\{2,3,4,10\},2),
\qquad
\Hst(X)=\{2,3,4,10\}
\]
is unique up to rotation.
\end{cor}

\begin{proof}
By Corollary~\ref{cor:23410-conditional}, such a set has seven points.
Hence Proposition~\ref{prop:23410-rigidity} applies.
\end{proof}

\begin{prob}
Decide Conjecture~\ref{conj:23410}.
\end{prob}

\begin{prob}
Determine \(N(\{p,q\},2)\) for distinct integers \(p,q>1\).
\end{prob}

\begin{prob}
Determine for which finite sets \(T\subset\NN\) the minimum-size set
\(X\subset S^1\) with \(\Hst(X)=T\) is unique up to rotation.
\end{prob}

\section*{Acknowledgements}
The authors would like to express their sincere gratitude to their supervisors,
Professor Tsuyoshi Miezaki and Professor Akihiro Munemasa, for their consistent
and kind guidance. Their helpful suggestions and many stimulating discussions
were essential to the completion of this work.

\section*{Statements and Declarations}

\paragraph{Funding}
Ryutaro Misawa was supported by the Japan Society for the Promotion of Science
through the JSPS Research Fellowship for Young Scientists and JSPS KAKENHI Grant Number JP26KJ0570.

\paragraph{Data availability}
Data sharing is not applicable to this article as no datasets were generated or
analysed during the current study.

\paragraph{Competing interests}
The authors have no competing interests to declare that are relevant to the
content of this article.

\paragraph{Author contributions}
Both authors contributed to the conception and development of the results and
to the writing and revision of the manuscript. Both authors read and approved
the final manuscript.

\end{document}